\documentclass[reqno]{amsart}
\usepackage{amscd}
\usepackage{amssymb}
\usepackage[margin=1in,includeheadfoot]{geometry}
\usepackage{cite}
\usepackage [latin1]{inputenc}
\usepackage{tabularx,booktabs,tikz}
\usepackage{caption}
\usepackage{amsmath}
\usepackage{amsfonts}

\usepackage{amscd}
\usepackage{amsthm}
\usepackage{amssymb} \usepackage{latexsym}
\usepackage{eufrak}
\usepackage{euscript}
\usepackage{epsfig}
\usepackage{graphics}
\usepackage{array}
\usepackage{enumerate}
\usepackage{dsfont}
\usepackage{color}
\usepackage{wasysym}
\usepackage{hyperref}
\usepackage{pdfsync}

\def\beq{\begin{equation}}
\def\eeq{\end{equation}}
\def\ba{\begin{array}}
\def\ea{\end{array}}

\def\R{\mathbb R}

\newtheorem{thm}{Theorem}[section]
\newtheorem{lm}[thm]{Lemma}

\newtheorem{crl}[thm]{Corollary}

\theoremstyle{definition}
\newtheorem{rem}[thm]{Remark}

\newtheorem{df}[thm]{Definition}

\theoremstyle{remark}

\begin{document}
\pagestyle{plain}
\title{$P$-Laplacian equations with general Choquard nonlinearity on lattice graphs}

\author{Lidan Wang}
\email{wanglidan@ujs.edu.cn}
\address{Lidan Wang: School of Mathematical Sciences, Jiangsu University, Zhenjiang 212013, People's Republic of China}

\begin{abstract}
In this paper, we study the following $p$-Laplacian equation
$$
-\Delta_{p} u+h(x)|u|^{p-2} u=\left(R_{\alpha} *F(u)\right)f(u)
$$
on lattice graphs $\mathbb{Z}^N$, where $p\geq 2$, $\alpha \in(0,N)$ are constants and $R_{\alpha}$ is the Green's function of the discrete fractional Laplacian that behaves as the Riesz potential. Under different assumptions on potential function $h$, we prove the existence of ground state solutions respectively by the methods of Nehari manifold.
\end{abstract}

\maketitle


\
\


\section{Introduction}

The $p$-Laplacian equation 
\begin{equation*}\label{a}
-\Delta_{p} u+h(x)|u|^{p-2} u=f(x, u), \quad x \in \mathbb{R}^N,
\end{equation*}
where $\Delta_pu=\text{div}(|\nabla u|^{p-2}\nabla u)$, has drawn lots of interest in recent years. In particular, for $p=2$,  this equation turns into the famous semilinear Schr\"{o}dinger equation, which has been studied extensively in the literature, see for examples \cite{DN,JT,LWZ,LW,R,SW}. For $p>2$, we refer the readers to \cite{AF,CT,EE,L1,LZ} for this equation.

Nowadays, many researchers turn to study differential equations on graphs, see for examples \cite{GJ,GJH,HW,HLW,LY}. The discrete analog of the above equation is
\begin{equation*}\label{b}
-\Delta_{p} u+h(x)|u|^{p-2} u=f(x, u), \quad x \in \mathbb{Z}^N,
\end{equation*}
where $\Delta_p u=\frac{1}{2} \sum\limits_{y \sim x}\left(|\nabla u|^{p-2}(y)+|\nabla u|^{p-2}(x)\right)(u(y)-u(x))$. The $p$-Laplacian on graphs was first introduced by \cite{NY} and has been well studied ever since. For $p=2$, see for examples  \cite{GLY,HLW,HX,W2,ZZ} for this discrete Schr\"{o}diner equation. For $p>2$, we refer the readers to \cite{M,HS,ZL} for this discrete $p$-Laplacian equation.
  
Very recently, the discrete Schr\"{o}dinger equations with Choquard-type  nonlinearity has attracted much attention. For example, Wang \cite{W1} and Wang et al. \cite{WZW} established the existence and asymptotic behaviors of ground state solutions to the Choquard equation on lattice graphs respectively. The authors in \cite{LW1, LZ1,LZ2} obtained the existence of ground state solutions to biharmonic Choquard equations on lattice graphs. Moreover, Wang \cite{W4} proved the existence of ground state solutions to a class of $p$-Laplacian equations with power Choquard nonlinearity on lattice graphs.
    
Inspired  by the papers above, in this paper, we study the following $p$-Laplacian equation with general convolution nonlinearity
\begin{equation}\label{aa}
-\Delta_{p} u+h(x)|u|^{p-2} u=\left(R_{\alpha} *F(u)\right)f(u) 
\end{equation}
on lattice graphs $\mathbb{Z}^N$, where $\Delta_{p} u(x)=\frac{1}{2} \sum\limits_{y \sim x}\left(|\nabla u|^{p-2}(y)+|\nabla u|^{p-2}(x)\right)(u(y)-u(x))$ with $p\geq 2$ and $R_\alpha$ with $\alpha \in(0, N)$ represents the Green's function of the discrete fractional Laplacian, $$ R_{\alpha}(x,y)=\frac{K_{\alpha}}{(2\pi)^N}\int_{\mathbb{T}^N}e^{i(x-y)\cdot k}\mu^{-\frac{\alpha}{2}}(k)\,dk,\quad x,y\in \mathbb{Z}^N,$$
which contains the fractional degree
$$K_\alpha=\frac{1}{(2\pi)^N}\int_{\mathbb{T}^N}\mu^{\frac{\alpha}{2}}(k)\,dk,\,\mu(k)=2N-2\sum\limits_{j=1}^N \cos(k_j),$$
where $\mathbb{T}^N=[0,2\pi]^N,\,k=(k_1,k_2,\dots,k_N)\in\mathbb{T}^N,$ and $\frac{2\pi}{N_j}\ell_j\rightarrow k_j$,\,$\ell_j=0,1,\dots,N_j-1,$ as $N_j\rightarrow\infty$. Clearly, the Green's function $R_{\alpha}$ has no singularity at $x=y$. By \cite{MC}, one sees that, for $|x-y|\gg 1$, $$R_\alpha\simeq|x-y|^{-(N-\alpha)},\quad N>\alpha.$$

Now we state the assumptions on potential $h$ and the nonlinearity $f$:
\begin{itemize}
\item[($h_1$)] for any $x\in\mathbb{Z}^N$, there exists a constant $h_0>0$ such that $h(x) \geq h_0$;
\item[($h_2$)] $h(x)$ is $\tau$-periodic for $x\in\mathbb{Z}^N$ with $\tau\in\mathbb{Z};$
\item[($h_3$)] there exists a point $x_0\in\mathbb{Z}^N$ such that $h(x)\rightarrow\infty$ as $|x-x_0|\rightarrow\infty;$
\item[($f_1$)] $f\in C(\mathbb{R},\mathbb{R})$ and $f(t)=o\left(|t|^{p-1}\right)$ as $t\rightarrow 0$;
\item[($f_2$)] there exist constants $C>0$ and $\tau>\frac{(N+\alpha)p}{2N}$ such that
$$
|f(t)|\leq C(1+|t|^{\tau-1}), \quad t\in \mathbb{R};
$$ 
\item[($f_3$)] there exists a constant $\theta>p$ such that $$0\leq \theta F(t)=\theta \int_{0}^{t}f(s)\,ds \leq 2f(t)t,\quad t\in \mathbb{R};$$
\item[($f_4$)] for any $u\in H\backslash\{0\}$,$$\frac{\int_{\mathbb{Z}^N}(R_\alpha \ast F(tu))f(tu)u\,d\mu}{t^p}$$ is strictly increasing with respect $t\in (0, \infty)$.
\end{itemize}

Obviously, by $(f_1)$ and $(f_2)$, one has that for any $\varepsilon>0$, there exists $C_\varepsilon>0$ such that
\begin{equation}\label{ad}
|f(t)|\leq\varepsilon|t|^{p-1}+C_\varepsilon|t|^{\tau-1},\quad t\in \mathbb{R},
\end{equation}
and hence
\begin{equation}\label{ae}
|F(t)|\leq\varepsilon|t|^{p}+C_\varepsilon|t|^{\tau},\quad t\in \mathbb{R}.
\end{equation}

Our main results are as follows.
\begin{thm}\label{t2}
 Assume that $(h_1)$, $(h_2)$ and $(f_1)$-$(f_4)$ hold. Then the equation (\ref{aa}) has a ground state solution.
\end{thm}

\begin{thm}\label{t1}
Assume that $(h_1)$, $(h_3)$ and $(f_1)$-$(f_4)$ hold. Then the equation (\ref{aa}) has a ground state solution.    
\end{thm}

This paper is organized as follows. In Section 2, we introduce some basic results on graphs. In Section 3, we prove Theorem \ref{t2} by the method of generalized Nehari manifold. In Section 4, we prove Theorem \ref{t1} by the mountain pass theorem and Nehari manifold.

\section{Preliminaries} 
In this section, we state some basic results on graphs. Let $G=(V, E)$ be a connected, locally finite graph, where $V$ denotes the vertex set and $E$ denotes the edge set. We call vertices $x$ and $y$ neighbors, denoted by $x \sim y$, if there exists an edge connecting them, i.e. $(x, y) \in E$. For any $x,y\in V$, the distance $d(x,y)$ is defined as the minimum number of edges connecting $x$ and $y$, namely
$$d(x,y)=\inf\{k:x=x_0\sim\cdots\sim x_k=y\}.$$
Let $B_{r}(a)=\{x\in V: d(x,a)\leq r\}$ be the closed ball of radius $r$ centered at $a\in V$. For brevity, we write $B_{r}:=B_{r}(0)$.

In this paper, we consider, the natural discrete model of the Euclidean space, the integer lattice graph.  The $N$-dimensional integer lattice graph, denoted by $\mathbb{Z}^N$, consists of the set of vertices $V=\mathbb{Z}^N$ and the set of edges $E=\{(x,y): x,\,y\in\mathbb{Z}^N,\,\underset {{i=1}}{\overset{N}{\sum}}|x_{i}-y_{i}|=1\}.$
In the sequel, we denote $|x-y|:=d(x,y)$ on the lattice graph $\mathbb{Z}^N$.

Let $C(\mathbb{Z}^N)$ be the set of all functions on $\mathbb{Z}^N$. For $u,v \in C(\mathbb{Z}^N)$, we define the gradient form $\Gamma$ as
$$
\Gamma(u, v)(x)=\frac{1}{2} \sum\limits_{y \sim x}(u(y)-u(x))(v(y)-v(x)):=\nabla u\nabla v.
$$
We write $\Gamma(u)=\Gamma(u, u)$ and denote the length of the gradient as
$$
|\nabla u|(x)=\sqrt{\Gamma(u)(x)}=\left(\frac{1}{2} \sum\limits_{y \sim x}(u(y)-u(x))^{2}\right)^{\frac{1}{2}}.
$$
Moreover, for $p\geq 2$, the $p$-Laplacian of $u\in C(\mathbb{Z}^N)$ is given by
$$
\Delta_{p} u(x)=\frac{1}{2} \sum_{y \sim x}\left(|\nabla u|^{p-2}(y)+|\nabla u|^{p-2}(x)\right)(u(y)-u(x)).
$$ Clearly, for $p=2$, we get the usual Laplacian on lattice graphs.

The space $\ell^{p}(\mathbb{Z}^N)$ is defined as $
\ell^{p}(\mathbb{Z}^N)=\left\{u \in C(\mathbb{Z}^N):\|u\|_{p}<\infty\right\},
$ where
$$
\|u\|_{p}= \begin{cases}\left(\sum\limits_{x \in \mathbb{Z}^N}|u(x)|^{p}\right)^{\frac{1}{p}}, &  1 \leq p<\infty, \\ \sup\limits_{x \in \mathbb{Z}^N}|u(x)|, & p=\infty.\end{cases}
$$
For $u\in C(\mathbb{Z}^N)$, in the following, we always write $$
\int_{\mathbb{Z}^N} f(x)\,d \mu=\sum\limits_{x \in \mathbb{Z}^N} f(x),
$$
where $\mu$ is the counting measure in $\mathbb{Z}^N$. 

Let $C_{c}(\mathbb{Z}^N)$ be the set of all functions on $\mathbb{Z}^N$ with finite support. Let $W^{1,p}(\mathbb{Z}^N)$ be the completion of $C_c(\mathbb{Z}^N)$ with respect to the norm
$$
\|u\|_{W^{1, p}(\mathbb{Z}^N)}=\left(\int_{\mathbb{Z}^N}\left(|\nabla u|^{p}+|u|^{p}\right) d \mu\right)^{\frac{1}{p}}.
$$ For $p>1$, $W^{1, p}(\mathbb{Z}^N)$ is a reflexive Banach space, see \cite{HLW}. Now we introduce a subspace
$$
H=\left\{u \in W^{1, p}(\mathbb{Z}^N): \int_{\mathbb{Z}^N} h(x)|u|^{p} d \mu<\infty\right\}
$$
equipped with the norm
$$
\|u\|=\left(\int_{\mathbb{Z}^N}\left(|\nabla u|^{p}+h(x)|u|^{p}\right)\,d \mu\right)^{\frac{1}{p}}.
$$
The space $H$ is a Hilbert space with the inner product
$$
(u, v)=\int_{\mathbb{Z}^N}\left(|\nabla u|^{p-2} \nabla u\nabla v+ h(x)|u|^{p-2} u v\right)\, d \mu.
$$
By $(h_1)$, we have $$\|u\|_p^p\leq \frac{1}{h_0}\int_{\mathbb{Z}^N} h(x)u^p(x)\,d\mu\leq \frac{1}{h_0}\|u\|^p.$$ Hence for any $u \in H$ and $q\geq p$, there holds
\begin{equation}\label{ac}
\|u\|_{q} \leq \|u\|_{p}\leq C \|u\|.
\end{equation}

The energy functional $J: H\rightarrow\R$ related to the equation (\ref{aa}) is given by
$$
J(u)=\frac{1}{p}\int_{\mathbb{Z}^N}\left(|\nabla u|^{p}+h(x)|u|^{p}\right) \,d\mu-\frac{1}{2}\int_{\mathbb{Z}^N}(R_\alpha \ast F(u))F(u)\,d\mu.
$$
Since $p\geq 2$, by the discrete HLS inequality (see Lemma \ref{lm1} below), one gets easily that $J\in C^{1}(H,\mathbb{R})$ and, for $\phi\in H$,
\begin{eqnarray*}
\left\langle J^{\prime}(u), \phi\right\rangle=\int_{\mathbb{Z}^N}\left(|\nabla u|^{p-2}\nabla u \nabla \phi+h|u|^{p-2} u \phi\right) \,d\mu-\int_{\mathbb{Z}^N}(R_\alpha \ast F(u))f(u)\phi\,d\mu.
\end{eqnarray*}

\begin{df}
$u \in H$ is a weak solution to the equation (\ref{aa}) if $u$ is a critical point of the functional $J$, i.e. $J'(u)=0$.
We say that $u$ is a ground state solution to the equation (\ref{aa}) if
$u$ is a nontrivial critical point of the functional $J$ such that
$$
J(u)=\inf\limits_{v\in \mathcal{M}} J(v):=c,
$$
where $\mathcal{M}=\{v\in H\backslash\{0\}: \langle J'(v),v\rangle=0\}$ is the corresponding Nehari manifold. 	
\end{df}

Next, we give some basic lemmas that useful in this paper. The first one is about the formulas of integration by parts, see \cite{HS}.

\begin{lm}\label{l0} 
Assume that $u \in W^{1, p}(\mathbb{Z}^N)$. Then for any $v \in C_{c}(\mathbb{Z}^N)$, we have
$$
\int_{\mathbb{Z}^N}|\nabla u|^{p-2} \nabla u \nabla v\,d \mu=-\int_{\mathbb{Z}^N}\left(\Delta_{p} u\right) v \,d \mu .
$$
\end{lm}

\begin{rem}
 If $u$ is a weak solution to the equation (\ref{aa}), then by Lemma \ref{l0}, one can get easily that $u$ is a pointwise solution to the equation (\ref{aa}).   
\end{rem}

The following discrete Hardy-Littlewood-Sobolev (HLS for abbreviation) inequality is  well-known, see \cite{LW1,W1}.
\begin{lm}\label{lm1}

Let $0<\alpha <N,\,1<r,s<\infty$ and $\frac{1}{r}+\frac{1}{s}+\frac{N-\alpha}{N}=2$. Then we have the discrete
HLS inequality
\begin{equation}\label{bo}
\int_{\mathbb{Z}^N}(R_\alpha\ast u)(x)v(x)\,d\mu\leq C_{r,s,\alpha,N}\|u\|_r\|v\|_s,\quad u\in \ell^r(\mathbb{Z}^N),\,v\in \ell^s(\mathbb{Z}^N).
\end{equation}
And an equivalent form is
\begin{eqnarray}\label{p1}
\|R_\alpha\ast u\|_{\frac{Nr}{N-\alpha r}}\leq C_{r,\alpha,N}\|u\|_r,\quad u\in \ell^r(\mathbb{Z}^N),
\end{eqnarray}
where $1<r<\frac{N}{\alpha}$.
\end{lm}

\begin{lm}\label{lhf}
Let $(f_1)$-$(f_3)$ hold. If $u\in H\backslash\{0\}$, then for $t\geq 1$, we have
\begin{equation}\label{95}
 \int_{\mathbb{Z}^N}(R_\alpha \ast F(tu))F(tu)\,d\mu\geq t^\theta\int_{\mathbb{Z}^N}(R_\alpha \ast F(u))F(u)\,d\mu.   
\end{equation}
\end{lm}
\begin{proof}
For $s>0$, let 
$$g(s)=\frac{1}{2}\int_{\mathbb{Z}^N}\left(R_\alpha \ast F(su)\right)F(su)\,d\mu.$$
By $(f_3)$, we get that
\begin{equation}\label{93}
\begin{aligned}
 g'(s)=&\int_{\mathbb{Z}^N}\left(R_\alpha \ast F(su)\right)f(su)u\,d\mu \\=&\frac{1}{s}\int_{\mathbb{Z}^N}\left(R_\alpha \ast F(su)\right)f(su)su\,d\mu\\ \geq& \frac{\theta}{2s}\int_{\mathbb{Z}^N}\left(R_\alpha \ast F(su)\right)F(su)\,d\mu\\=& \frac{\theta}{s} g(s).
 \end{aligned}
\end{equation}

Clearly for $t=1$, the result (\ref{95}) holds.
For $t>1$, integrating the inequality (\ref{93}) from $1$ to $t$, one gets easily that
$$g(t)\geq  t^\theta g(1),$$
which is equivalent to 
$$\int_{\mathbb{Z}^N}(R_\alpha \ast F(tu))F(tu)\,d\mu\geq t^\theta\int_{\mathbb{Z}^N}(R_\alpha \ast F(u))F(u)\,d\mu. $$
\end{proof}

Finally, we give two compactness results for $H$. The first one is a discrete Lions lemma, which plays a crucial role in our first theorem.
\begin{lm}\label{lgh}
 Let $p\leq q<\infty$. If $\left\{u_{n}\right\}$ is bounded in $H$ and
$$
\left\|u_{n}\right\|_{\infty} \rightarrow 0,\quad n \rightarrow\infty,
$$
then, for any $p<q<\infty$,
$$
u_{n} \rightarrow 0,\quad  \text { in } \ell^{q}(\mathbb{Z}^N) \text {. }
$$   
\end{lm} 

\begin{proof}
 It follows from (\ref{ac}) that $\{u_n\}$ is bounded in $\ell^{p}(\mathbb{Z}^N)$. Hence, for $p<q<\infty$, by an interpolation inequality, we get that
$$
\left\|u_{n}\right\|_{q}^{q} \leq\left\|u_{n}\right\|_{p}^{p}\left\|u_{n}\right\|_{\infty}^{q-p}\rightarrow 0,\quad n\rightarrow\infty.
$$
\end{proof}

The second one is a well-known compact embedding result, we refer the readers to \cite{ZZ}.

\begin{lm}\label{lgg}
Let $(h_1)$ and $(h_2)$ hold. Then for any $q\geq p\geq 2$, $H$ is continuously embedded into $\ell^{q}(\mathbb{Z}^N)$. That is, there exists a constant $C$ depending only on $q$ such that, for any $u \in H$,
$$
\|u\|_{q} \leq C\|u\|.
$$
Moreover, for any bounded sequence $\left\{u_{n}\right\} \subset H$, there exists $u \in H$ such that, up to a subsequence, 
$$
\begin{cases}u_{n} \rightharpoonup u, & \text { in } H, \\ u_{n} \rightarrow u, & \text{pointwise~in}~ \mathbb{Z}^N, \\ u_{n} \rightarrow u, & \text { in } \ell^{q}(\mathbb{Z}^N) .\end{cases}
$$ 
\end{lm}

\
\

\section{Proof of Theorem \ref{t2}}
In this section, we prove the existence of ground state solutions to the equation (\ref{aa}) under the assumptions $(h_1)$ and $(h_2)$ on potential function $h$. 

The condition $(h_2)$ leads to the loss of compactness in $H$. Moreover, since we  only assume that $f$ is continuous, $\mathcal{M}$ is not a $C^1$-manifold. So we cannot use the Ekeland variational principle on $\mathcal{M}$ directly. In order to overcome the difficulties , we shall prove Theorem \ref{t2} by the generalized Nehari manifold.

First, we show some properties of $J$ on the Nehari manifold $\mathcal{M}$ that are useful in our proofs.

\begin{lm}\label{lj}
 Let $\left(h_{1}\right)$ and $\left(f_{1}\right)$-$\left(f_{4}\right)$ hold. Then 
\begin{itemize}
    \item[(i)] for any $u \in H \backslash\{0\}$, there exists a unique $s_{u}>0$ such that $s_{u} u \in \mathcal{M}$ and $J(s_{u} u)=$ $\max\limits_{s>0} J(s u)$;
    
    \item[(ii)] there exists $\eta>0$ such that $\|u\| \geq \eta$ for  $u \in \mathcal{M}$, and hence $c=\inf\limits_{\mathcal{M}}J>0$;

    \item[(iii)] $J$ is bounded from below on $\mathcal{M}$ by a positive constant.


\end{itemize}
  
\end{lm}

\begin{proof}(i) Let $u\in H\backslash\{0\}$ be fixed and $s>0$.
Since $\alpha\in(0,N)$ and $\tau>\frac{(N+\alpha)p}{2N}$,
by (\ref{ae}), (\ref{ac}) and the HLS inequality (\ref{bo}), we get that
\begin{eqnarray}\label{mc}
    \int_{\mathbb{Z}^N}(R_\alpha \ast F(su))F(su)\,d\mu\nonumber&\leq& C\left(\int_{\mathbb{Z}^N} |F(su)|^{\frac{2N}{N+\alpha}}\,d\mu  \right)^{\frac{N+\alpha}{N}}\\&\leq &C\left(\int_{\mathbb{Z}^N} \left(\varepsilon|su|^p+C_\varepsilon|su|^{\tau}\right)^{\frac{2N}{N+\alpha}}\,d\mu  \right)^{\frac{N+\alpha}{N}}\nonumber\\&\leq& \varepsilon s^{2p}\|u\|^{2p}_{\frac{2Np}{N+\alpha}}+C_\varepsilon s^{2\tau}\|u\|^{2\tau}_{\frac{2N\tau}{N+\alpha}}\nonumber\\&\leq& \varepsilon s^{2p}\|u\|^{2p}+C_\varepsilon s^{2\tau}\|u\|^{2\tau}.
\end{eqnarray}
Then we have
\begin{equation}\label{jv}
\begin{aligned}
J(s u) & =\frac{s^{p}}{p}\|u\|^{p}-\frac{1}{2}\int_{\mathbb{Z}^N}(R_\alpha \ast F(su))F(su)\,d\mu \\&\geq\frac{s^{p}}{p}\|u\|^{p}-\varepsilon s^{2p}\|u\|^{2p}-C_\varepsilon s^{2\tau}\|u\|^{2\tau}.
\end{aligned}
\end{equation}
Note that $2\tau>\frac{(N+\alpha)p}{N}>p$. Let $\varepsilon\rightarrow 0^+$, we get easily that $J(s u)>0$ for $s>0$ sufficiently small.

On the other hand, note that $\theta>p$, by (\ref{95}), we obtain that
\begin{eqnarray}\label{md}
 \lim _{s\rightarrow\infty} J(su)\nonumber&=&\lim _{s \rightarrow\infty}\left[\frac{s^{p}}{p}\|u\|^p-\frac{1}{2} \int_{\mathbb{Z}^N}(R_\alpha \ast F(su))F(su)\,d\mu\right]\nonumber\\&\leq&\lim _{s\rightarrow\infty}\left[\frac{s^{p}}{p}\|u\|^p-\frac{s^\theta}{2} \int_{\mathbb{Z}^N}\left(R_\alpha \ast F(u)\right)F(u)\,d\mu\right]\nonumber\\&\rightarrow&-\infty, \quad s \rightarrow \infty.  \end{eqnarray}
Hence $\max\limits_{s>0} J(s u)$ is achieved at some $s_{u}>0$ with $s_{u} u \in \mathcal{M}$. 

Finally, we prove the uniqueness of $s_{u}$. If there exist $s_{u}^{\prime}>s_{u}>0$ such that $s_{u}^{\prime} u\in\mathcal{M}$ and $s_{u} u \in \mathcal{M}$, then we have
$$
\begin{aligned}
& \frac{1}{\left(s_{u}^{\prime}\right)}\|u\|^p=\int_{\mathbb{Z}^N}\frac{(R_\alpha \ast F(s'_uu))f(s'_uu)u}{\left(s_{u}^{\prime}\right)^{p}}\,d\mu, \\
& \frac{1}{\left(s_{u}\right)}\|u\|^p=\int_{\mathbb{Z}^N}\frac{(R_\alpha \ast F(s_uu))f(s_uu)u}{\left(s_{u}\right)^{p}}\,d\mu,
\end{aligned}
$$
and hence
$$
\begin{aligned}
\left(\frac{1}{\left(s_{u}^{\prime}\right)}-\frac{1}{\left(s_{u}\right)}\right)\|u\|^p=\int_{\mathbb{Z}^N}\frac{(R_\alpha \ast F(s'_uu))f(s'_uu)u}{\left(s_{u}^{\prime}\right)^{p}}\,d\mu-\int_{\mathbb{Z}^N}\frac{(R_\alpha \ast F(s_uu))f(s_uu)u}{\left(s_{u}\right)^{p}}\,d\mu.
\end{aligned}
$$
By $\left(f_{4}\right)$, we get a contradiction.

\
\

(ii) By the HLS inequality (\ref{bo}), we get that
\begin{eqnarray}\label{aj}
\int_{\mathbb{Z}^N}(R_\alpha \ast F(u))f(u)u\nonumber\,d\mu
&\leq& C\left(\int_{\mathbb{Z}^N} |F(u)|^{\frac{2N}{N+\alpha}}\,d\mu  \right)^{\frac{N+\alpha}{2N}}\left(\int_{\mathbb{Z}^N} |f(u)u|^{\frac{2N}{N+\alpha}}\,d\mu  \right)^{\frac{N+\alpha}{2N}}\nonumber\\&\leq&C\left(\int_{\mathbb{Z}^N} \left(\varepsilon|u|^p+C_\varepsilon|u|^{\tau}\right)^{\frac{2N}{N+\alpha}}\,d\mu  \right)^{\frac{N+\alpha}{N}}\nonumber\\&\leq& \varepsilon\|u\|^{2p}_{\frac{2Np}{N+\alpha}}+C_\varepsilon\|u\|^{2\tau}_{\frac{2N\tau}{N+\alpha}}\nonumber\\&\leq& \varepsilon\|u\|^{2p}+C_\varepsilon\|u\|^{2\tau}.
\end{eqnarray}
Then for $u \in\mathcal{M}$, one has that
$$
\begin{aligned}
0&=\left\langle J^{\prime}(u), u\right\rangle\\ & =\|u\|^{p}-\int_{\mathbb{Z}^N}(R_\alpha \ast F(u))f(u)u\,d\mu  \\
& \geq\|u\|^{p}-\varepsilon\|u\|^{2p}-C_{\varepsilon}\|u\|^{2\tau} .
\end{aligned}
$$
Since $2\tau>p$, we get easily that there exists $\eta>0$ such that $\|u\| \geq \eta>0$.

\
\

(iii) For any $u \in \mathcal{M}$, by $(f_3)$ and (ii), one has that
\begin{eqnarray*}
J(u) &=& J(u)-\frac{1}{\theta}\left\langle J^{\prime}(u), u\right\rangle \\
&=&\left(\frac{1}{p}-\frac{1}{\theta}\right)\|u\|^{p}+\frac{1}{2}\int_{\mathbb{Z}^N}(R_\alpha \ast F(u))\left(\frac{2}{\theta}f(u)u-F(u)\right)\,d\mu\\
& \geq &\left(\frac{1}{p}-\frac{1}{\theta}\right)\|u\|^{p} \\& \geq &\left(\frac{1}{p}-\frac{1}{\theta}\right) \eta^{p}\\&>&0.
\end{eqnarray*}

\end{proof}

Now we can establish a  homeomorphic map between the unit sphere $S\subset H$ and the Nehari manifold $\mathcal{M}$, which states as follows.

\begin{lm}\label{lg}
  Assume that $\left(h_{1}\right)$ and $\left(f_{1}\right)$-$\left(f_{4}\right)$ hold. Define the maps $s:H\backslash\{0\}\rightarrow (0,\infty)$, $u\mapsto s_u$ and $$
\begin{aligned}
\widehat{m}: H\backslash\{0\} & \rightarrow \mathcal{M}, \\
u & \mapsto \widehat{m}(u)=s_{u} u.
\end{aligned}
$$ 
Then we have the following results,
\begin{itemize}
    \item [(i)] the maps $s$ and $\widehat{m}$ are continuous;
    \item[(ii)] the map $m:=\widehat{m}\mid_{S}$ is a homeomorphism between $S$ and $\mathcal{M}$, and the inverse of $m$ is given by
\begin{equation*}
m^{-1}(u)=\frac{u}{\|u\|}. 
\end{equation*}
\end{itemize}
\end{lm}
\begin{proof}
(i) Let $u_n\rightarrow u$ in $H\backslash\{0\}$. Denote $s_n=s_{u_n}$, then $\widehat{m}(u_n)=s_n u_n\in\mathcal{M}$. We first prove that for any $t>0$, $\widehat{m}(t u)=\widehat{m}(u)$. Indeed, since $\hat{m}(tu)=s_{tu}tu\in \mathcal{M}$ and $\hat{m}(u)=s_u u\in\mathcal{M}$, we get respectively that
$$\|u\|^p=\int_{\mathbb{Z}^N}(R_\alpha \ast F(s_{tu}tu))f(s_{tu}tu)\frac{u}{(s_{tu}t)^{p-1}}\,d\mu,$$
and
$$\|u\|^p=\int_{\mathbb{Z}^N}(R_\alpha \ast F(s_uu))f(s_uu)\frac{u}{s_u^{p-1}}\,d\mu.$$
The two formulas above imply $s_u=s_{tu}t$, and hence
$$\hat{m}(tu)=s_{tu}tu=s_uu=\hat{m}(u).$$
Therefore, without loss of generality, we may assume that $\{u_n\}\subset S$. By Lemma \ref{lj} (ii), we get that $$s_n=\|s_nu_n\|\geq\eta>0.$$ 
We claim that $\{s_n\}$ is bounded. Otherwise, $s_n\rightarrow\infty$ as $n\rightarrow\infty$. 
By (\ref{95}), one has that $$\int_{\mathbb{Z}^N}(R_\alpha \ast F(s_nu_n))F(s_nu_n)\,d\mu\geq s_n^\theta\int_{\mathbb{Z}^N}\left(R_\alpha \ast F\left(u_n\right)\right)F\left(u_n\right)\,d\mu.$$
Moreover, since $\|u_n\|=1$, we get easily that
$$\int_{\mathbb{Z}^N}\left(R_\alpha \ast F\left(u_n\right)\right)F\left(u_n\right)\,d\mu\leq C.$$
Then it follows from (iii) of Lemma \ref{lj} and   $\theta>p$ that
\begin{eqnarray*}
 0 &<& \frac{J\left(s_nu_{n}\right)}{\left\|s_nu_{n}\right\|^{p}}\\&=&\frac{1}{p}-\frac{\int_{\mathbb{Z}^N}(R_\alpha \ast F(s_nu_n))F(s_nu_n)\,d\mu}{\|s_nu_{n}\|^p}\\&\leq & \frac{1}{p}-s_n^{\theta-p}\int_{\mathbb{Z}^N}\left(R_\alpha \ast F\left(u_n\right)\right)F\left(u_n\right)\,d\mu\\& \rightarrow&-\infty.  \end{eqnarray*}
 We get a contradiction. Hence  $\{s_n\}$ is bounded. Then, up to a subsequence, there exists $s_0>0$ such that $s_n\rightarrow s_0$ and $\widehat{m}(u_n)\rightarrow s_0 u$. Since $\mathcal{M}$ is closed, we get that $s_0u\in\mathcal{M}$, which implies that $s_0=s_u$. Hence
 $$s_{u_n}\rightarrow s_u$$
 and
 $$\widehat{m}(u_n)\rightarrow s_0u=s_uu=\widehat{m}(u).$$
Thus the maps $s$ and $\widehat{m}$ are continuous.

\
\

 (ii) Obviously, $m$ is continuous. We first prove that $m$ is surjective.  In fact, for any $u\in\mathcal{M}$, denote $\tilde{u}=\frac{u}{\|u\|}$, then $\tilde{u}\in S.$ By Lemma \ref{lj} (i), there exists a unique $s_{\tilde{u}}$ such that $s_{\tilde{u}}\tilde{u}\in \mathcal{M}$. Note that $u=\|u\|\tilde{u}\in \mathcal{M}$. By the uniqueness of $s_{\tilde{u}}$, we get $s_{\tilde{u}}=\|u\|.$ Hence $m(\tilde{u})=s_{\tilde{u}}\tilde{u}=u\in\mathcal{M}$.  
 
 Next, we show that $m$ is injective. Indeed, let $u_1,u_2\in S$ and $u_1\neq u_2$.  If $s_{u_1}= s_{u_2}$, clearly $s_{u_1}u_1\neq s_{u_2}u_2$. If $s_{u_1}\neq s_{u_2}$, then $\|s_{u_1}u_1\|=s_{u_1}\neq s_{u_2}=\|s_{u_2}u_2\|$, which implies that $s_{u_1}u_1\neq s_{u_2}u_2$. 
 
 Hence $m$ has an inverse mapping $m^{-1}:\mathcal{M}\rightarrow S$ with $m^{-1}(u)=\frac{u}{\|u\|}.$ 
    
\end{proof}

By Lemma \ref{lj} and Lemma \ref{lg}, we have the following corollary.
\begin{crl}\label{jb}
Let $\left(h_{1}\right)$ and $\left(f_{1}\right)$-$\left(f_{4}\right)$ hold. Then we have
$$\inf\limits_{u\in\mathcal{M}} J(u)=\inf\limits_{u\in H \backslash\{0\}} \max _{s>0} J(su)=\inf\limits_{\gamma\in\Gamma}\max\limits_{s\in [0,1]} J(\gamma(s)),$$
where $$\Gamma=\{\gamma\in C([0,1],H):\gamma(0)=0, J(\gamma(1))<0\}.$$ 
\end{crl}
\begin{proof}
The proof is similar to that of \cite[Lemma 4.3]{W3}. We omit it here.

\end{proof}

The following lemma shows that $\Psi$ (see below) is of  class $C^{1}$ and there is a one-to-one correspondence between critical points of $\Psi$ and nontrivial critical points of $J$. 

\begin{lm}\label{lgk}
Assume that $\left(h_{1}\right)$ and $\left(f_{1}\right)$-$\left(f_{4}\right)$ hold. Define the functional
  $$
  \begin{aligned}
\Psi: S &\rightarrow \mathbb{R},\\
w & \mapsto \Psi(w)=J(m(w)).
\end{aligned}
$$ 
Then we have
\begin{itemize}
    \item [(i)] $\Psi(w) \in C^{1}(S, \mathbb{R})$ and $$
\langle \Psi^{\prime}(w),z\rangle=\|m(w)\|\left\langle J^{\prime}(m(w)), z\right\rangle, \quad z \in T_{w}(S)=\{v\in H:(w, v)=0\};
$$
\item [(ii)] $\left\{w_{n}\right\}$ is a Palais-Smale sequence for $\Psi$ if and only if $\left\{m(w_{n})\right\}$ is a Palais-Smale sequence for $J$;

\item[(iii)] 
$w\in S$ is a critical point of $\Psi$ if and only if $m(w) \in \mathcal{M}$ is a nontrivial critical point of $J$. Moreover, the corresponding critical values of $\Psi$ and $J$ coincide and $
\inf\limits_{S} \Psi=\inf\limits_{\mathcal{M}}J.
$
\end{itemize}
 
\end{lm} 

\begin{proof}
    With some modifications in \cite[Lemma 5.1]{W3}, we can get the desired result. 
\end{proof}

\
\

{\bf Proof of Theorem \ref{t2}. }  By Lemma \ref{lgk} (iii), one gets that $c=\inf\limits_{S}\Psi$. Let $\left\{w_{n}\right\} \subset S$ be a minimizing sequence satisfying $\Psi(w_{n}) \rightarrow c$. By Ekeland's variational principle, we may assume that $\Psi^{\prime}(w_{n}) \rightarrow 0$ as $n \rightarrow \infty$. Then one gets that $\left\{w_{n}\right\}$ is a $(PS)_c$ sequence of $\Psi.$ 

Let $u_{n}=m(w_{n}) \in \mathcal{M}$. Then it follows from Lemma \ref{lgk} (ii) that $$J\left(u_{n}\right) \rightarrow c,\qquad\text{and}\qquad J^{\prime}\left(u_{n}\right) \rightarrow 0, \quad n \rightarrow \infty.$$  
Note that $\theta>p$. Then we get that
\begin{eqnarray}\label{ba}
\|u_n\|^p\nonumber&=&\frac{p}{2}\int_{\mathbb{Z}^N}(R_\alpha \ast F(u_n))F(u_n)\,d\mu+pc+o_n(1)\nonumber\\&\leq& \frac{p}{\theta}\int_{\mathbb{Z}^N}(R_\alpha \ast F(u_n))f(u_n)u_n\,d\mu+pc+o_n(1)\nonumber\\&\leq&\frac{p}{\theta} \left(\|u_n\|^p+o_n(1)\|u_n\|\right)+pc+o_n(1)\nonumber\\&=&\frac{p}{\theta}\|u_n\|^p+o_n(1)\|u_n\|+pc+o_n(1),
\end{eqnarray}
where $o_n(1)\rightarrow 0$ as $n\rightarrow\infty.$
This inequality means that $\{u_n\}$ is bounded in $H$.
 Hence there exists $u \in H$ such that $$u_{n} \rightharpoonup u,\quad\text{in~}H, \qquad\text{and}\qquad u_n\rightarrow u,\quad\text{pointwise~in}~\mathbb{Z}^N.$$

If \begin{equation}\label{hq}
 \left\|u_{n}\right\|_{\infty} \rightarrow 0,\quad n\rightarrow\infty,   
\end{equation} 
then by Lemma \ref{lgh}, we get that $u_{n} \rightarrow 0$ in $\ell^{q}(\mathbb{Z}^N)$ with $q>p$, which makes  
\begin{eqnarray*}
    \int_{\mathbb{Z}^N}(R_\alpha \ast F(u_n))f(u_n)u_n\,d\mu&\leq& C\left( \|u_n\|^{2p}_{\frac{2Np}{N+\alpha}}+\|u_n\|^{2\tau}_{\frac{2N\tau}{N+\alpha}}\right)\\&\rightarrow& 0,\quad n\rightarrow\infty.
\end{eqnarray*}
Namely
$$\int_{\mathbb{Z}^N}(R_\alpha \ast F(u_n))f(u_n)u_n\,d\mu=o_{n}(1).$$  Since $\{u_n\}\subset\mathcal{M}$, one has that
$$
\begin{aligned}
0=\left\langle J^{\prime}\left(u_{n}\right), u_{n}\right\rangle & =\left\|u_{n}\right\|^{p}-\int_{\mathbb{Z}^N}(R_\alpha \ast F(u_n))f(u_n)u_n\,d\mu\\
& =\left\|u_{n}\right\|^{p}+o_{n}(1),
\end{aligned}
$$
which implies that $\left\|u_{n}\right\| \rightarrow 0$ as $n \rightarrow \infty$. This contradicts $\left\|u_{n}\right\| \geq \eta>0$. Thus (\ref{hq}) does not hold, and hence there exists $\delta>0$ such that
\begin{equation}\label{hw}
\liminf _{n \rightarrow \infty}\left\|u_{n}\right\|_{\infty} \geq \delta>0.
\end{equation}
This means that $u\neq 0$ and there exists a sequence $\left\{y_{n}\right\} \subset \mathbb{Z}^N$ such that
\begin{equation*}
\left|u_{n}(y_{n})\right| \geq \frac{\delta}{2}.
\end{equation*}
Let $\{k_{n}\}\subset\mathbb{Z}^N$ satisfy $\left\{y_{n}-k_{n} \tau\right\} \subset \Omega$, where $\Omega=[0, \tau)^{N}$. Let $v_{n}(y):=u_{n}\left(y+k_{n} \tau\right)$, then for $\{v_{n}\}$,
\begin{equation*}
\left\|v_{n}\right\|_{l^{\infty}(\Omega)} \geq\left|v_{n}\left(y_{n}-k_{n} \tau\right)\right|=\left|u_{n}(y_{n})\right| \geq \frac{\delta}{2}>0.
\end{equation*}
Since $h$ is $\tau$-periodic, $\mathcal{M}$ and $J$ are invariant under the translation, we obtain that $\left\{v_{n}\right\}\subset\mathcal{M}$ and it is also a $(PS)_c$ sequence for $J$, 
\begin{equation}\label{98}
J\left(v_{n}\right) \rightarrow c,\qquad\text{and}\qquad J^{\prime}\left(v_{n}\right) \rightarrow 0, \quad n \rightarrow \infty.    
\end{equation}
Moreover, $\{v_n\}$ is bounded in $H$. Then we can find a $v\in H$ with $v\neq 0$ such that 
\begin{equation}\label{96}
 v_{n} \rightharpoonup v,\quad\text{in~}H, \qquad\text{and}\qquad v_n\rightarrow v,\quad \text{pointwise~in~}\mathbb{Z}^N.  \end{equation}

In the following, we prove that $v$ is a ground state solution to the equation (\ref{aa}). We first claim that 
\begin{equation}\label{99}
 \int_{\mathbb{Z}^N}(R_\alpha \ast F(v_n))f(v_n)\phi\,d\mu\rightarrow \int_{\mathbb{Z}^N}(R_\alpha \ast F(v))f(v)\phi\,d\mu, \quad \phi\in C_c(\mathbb{Z}^N).  
\end{equation}
In fact, let $\text{supp}(\phi)\subset B_r$ with $r>1$. A direct calculation yields that
 $$
 \begin{aligned}
 &\left|\int_{\mathbb{Z}^N}(R_\alpha \ast F(v_n))f(v_n)\phi\,d\mu- \int_{\mathbb{Z}^N}(R_\alpha \ast F(v))f(v)\phi\,d\mu\right|\\=&\left|\int_{\mathbb{Z}^N}(R_\alpha \ast \left(F(v_n)-F(v)\right)f(v)\phi\,d\mu\right|+\left|\int_{\mathbb{Z}^N}(R_\alpha \ast F(v_n))\left(f(v_n)-f(v)\right)\phi\,d\mu\right|.
 \end{aligned}
 $$
 
For the first term on the right hand side, by (\ref{ae}) and (\ref{ac}), we obtain that $\{F(u_n)\}$ is bounded in $\ell^{\frac{2N}{N+\alpha}}(\mathbb{Z}^N)$. Then it follows from the HLS inequality (\ref{p1}) that $\{(R_\alpha \ast F(u_n))\}$ is bounded in $\ell^{\frac{2N}{N-\alpha}}(\mathbb{Z}^N)$. Moreover, we have $F(u_n)\rightarrow F(u)$ pointwise in $\mathbb{Z}^N$. By passing to a subsequence, we have
$$(R_\alpha \ast F(v_n))\rightharpoonup (R_\alpha \ast F(v)), \quad\text{in~}\ell^{\frac{2N}{N-\alpha}}(\mathbb{Z}^N).$$ 
Note that $f(v)\phi\in\ell^{\frac{2N}{N+\alpha}}(\mathbb{Z}^N) $, then $$\left|\int_{\mathbb{Z}^N}(R_\alpha \ast \left(F(v_n)-F(v)\right)f(v)\phi\,d\mu\right|\rightarrow0,\quad n\rightarrow\infty.$$

For the second term on the right hand side, by the HLS inequality (\ref{bo}), we obtain that
$$
\begin{aligned}
  &\left|\int_{\mathbb{Z}^N}(R_\alpha \ast F(v_n))\left(f(v_n)-f(v)\right)\phi\,d\mu\right|\\\leq & C\|F(v_n)\|_{\frac{2N}{N+\alpha}}\left(\int_{\mathbb{Z}^N}|(f(v_n)-f(v))\phi|^{\frac{2N}{N+\alpha}}\,d\mu\right)^{\frac{N+\alpha}{2N}}\\\leq& C\left(\int_{\mathbb{Z}^N}|(f(v_n)-f(v))\phi|^{\frac{2N}{N+\alpha}}\,d\mu\right)^{\frac{N+\alpha}{2N}}\\=&C\left(\int_{B_r}|(f(v_n)-f(v))\phi|^{\frac{2N}{N+\alpha}}\,d\mu\right)^{\frac{N+\alpha}{2N}}\\ \rightarrow& 0,\quad n\rightarrow\infty. 
\end{aligned}
$$
Hence we get (\ref{99}). Since $B_r$ is a finite set in $\mathbb{Z}^N$, by (\ref{96}) and (\ref{99}), we derive that
$$
\begin{aligned}
   \langle J'(v_n),\phi\rangle=&\int_{\mathbb{Z}^N}\left(|\nabla v_n|^{p-2}\nabla v_n \nabla \phi+h|v_n|^{p-2} v_n \phi\right) \,d\mu-\int_{\mathbb{Z}^N}(R_\alpha \ast F(v_n))f(v_n)\phi\,d\mu\\=&\int_{B_r\cup \partial B_r}|\nabla v_n|^{p-2}\nabla v_n \nabla \phi\,d\mu+\int_{B_r}h|v_n|^{p-2} v_n \phi \,d\mu-\int_{B_r}(R_\alpha \ast F(v_n))f(v_n)\phi\,d\mu\\ \rightarrow & \int_{B_r\cup \partial B_r}|\nabla v|^{p-2}\nabla v \nabla \phi\,d\mu+\int_{B_r}h|v|^{p-2} v\phi\,d\mu-\int_{B_r}(R_\alpha \ast F(v))f(v)\phi\,d\mu\\=&\int_{\mathbb{Z}^N}\left(|\nabla v|^{p-2}\nabla v \nabla \phi+h|v|^{p-2} v \phi\right) \,d\mu-\int_{\mathbb{Z}^N}(R_\alpha \ast F(v))f(v)\phi\,d\mu\\=& \langle J'(v),\phi\rangle,
\end{aligned}
$$
where $\partial B_r:=\{y\in\mathbb{Z}^N\backslash B_r :\exists~x\in B_r~ \text{such~that}~y\sim x \}$ is the vertex boundary of $B_r$. Combined with (\ref{98}), we get that $$0\leq \left|\langle J'(v),\phi\rangle\right|=\lim\limits_{n\rightarrow\infty}\left|\langle J'(v_n),\phi\rangle\right|\leq \lim\limits_{n\rightarrow\infty}\|J'(v_n)\|\|\phi\|=0.$$
This means that $\langle J'(v),\phi\rangle=0$. Since $C_c(\mathbb{Z}^N)$ is dense in $H$, we obtain that for any $\phi\in H$,
 $\langle J'(v),\phi\rangle=0$, and hence $v\in \mathcal{M}$.
It remains to prove $J(v)=c$. 
By $(f_3)$, we have that
\begin{equation}\label{97}
 \left(\frac{1}{p}f(v_n)v_n-\frac{1}{2}F(v_n)\right)>\frac{1}{2}\left(\frac{2}{\theta}f(v_n)v_n-F(v_n)\right)\geq 0.   
\end{equation}
Then combing (\ref{98}), (\ref{96}), (\ref{97}) and Fatou's lemma, we get that
$$
\begin{aligned}
c\leq& J(v)=J(v)-\frac{1}{p}\langle J'(v),v\rangle\\=& \frac{1}{p}\int_{\mathbb{Z}^N}(R_\alpha \ast F(v))f(v)v\,d\mu-\frac{1}{2}\int_{\mathbb{Z}^N}(R_\alpha \ast F(v))F(v)\,d\mu\\=&\int_{\mathbb{Z}^N}(R_\alpha \ast F(v))\left(\frac{1}{p}f(v)v-\frac{1}{2}F(v)\right)\,d\mu\\ \leq &\int_{\mathbb{Z}^N}(R_\alpha \ast F(v_n))\left(\frac{1}{p}f(v_n)v_n-\frac{1}{2}F(v_n)\right)\,d\mu\\=&\frac{1}{p}\int_{\mathbb{Z}^N}(R_\alpha \ast F(v_n))f(v_n)v_n\,d\mu-\frac{1}{2}\int_{\mathbb{Z}^N}(R_\alpha \ast F(v_n))F(v_n)\,d\mu\\=&
\lim\limits_{n\rightarrow\infty} J(v_n)-\frac{1}{p}\langle J'(v_n),v_n\rangle\\=&\lim\limits_{n\rightarrow\infty} J(v_n)\\=& c.
\end{aligned}
$$
Hence $J(v)=c$. The proof is completed.

 \qed

\
\

\section{Proof of Theorem \ref{t1}}
In this section, we prove the existence of ground state solutions to the equation  (\ref{aa}) by the mountain pass theorem and the method of Nehari manifold.  

Recall that, for a given functional $\Phi\in C^{1}(X,\mathbb{R})$, a sequence $\{u_n\}\subset X$ is a Palais-Smale sequence at level $c\in\mathbb{R}$, $(PS)_c$ sequence for short, of the functional $\Phi$, if it satisfies, as $n\rightarrow\infty$,
\begin{eqnarray*}
\Phi(u_n)\rightarrow c, \qquad \text{in}~ X,\qquad\text{and}\qquad
\Phi'(u_n)\rightarrow 0, \qquad \text{in}~X^*,
\end{eqnarray*}
where $X$ is a Banach space and $X^{*}$ is the dual space of $X$. Moreover, we say that $\Phi$ satisfies $(PS)_c$ condition, if any $(PS)_c$ sequence has a convergent subsequence. 

First we show that
 the functional $J$ satisfies the mountain pass geometry.
\begin{lm}\label{lm}
  Assume that $(h_1)$ and $(f_1)$-$(f_3)$ hold. Then
\begin{itemize}
    \item[(i)] there exist $\sigma, \rho>0$ such that $J(u) \geq \sigma>0$ for $\|u\|=\rho$;
    \item[(ii)] there exists $e \in H$ with $\|e\|>\rho$ such that $J(e)< 0$.  
\end{itemize}  
\end{lm}
\begin{proof}
  (i) By similar arguments to (\ref{mc}), one has that
\begin{eqnarray*}\label{uc}
    \int_{\mathbb{Z}^N}(R_\alpha \ast F(u))F(u)\,d\mu\leq \varepsilon\|u\|^{2p}+C_\varepsilon\|u\|^{2\tau},
\end{eqnarray*}
and hence
\begin{align*}
J(u) & =\frac{1}{p}\|u\|^p-\frac{1}{2}  \int_{\mathbb{Z}^N}(R_\alpha \ast F(u))F(u)\,d\mu\\
& \geq \frac{1}{p}\|u\|^p-\varepsilon\|u\|^{2p}-C_\varepsilon\|u\|^{2\tau}.
\end{align*}
Note that $2\tau>\frac{(N+\alpha)p}{N}>p$. Let $\varepsilon\rightarrow 0^+$,  then there exist $\sigma, \rho>0$ small enough such that $J(u) \geq \sigma>0$ for $\|u\|=\rho$.

\
\

(ii) Let $u\in H\backslash\{0\}$ be fixed. Then it follows from (\ref{md}) that
\begin{eqnarray*}\label{ud}
 \lim _{t \rightarrow\infty} J(t u)\rightarrow -\infty.   
\end{eqnarray*}
Hence, there exists $t_{0}>0$ large enough such that $\left\|t_{0} u\right\|>\rho$ and $J\left(t_{0} u\right)<0$. We get the result by letting $e=t_{0} u$. 
\end{proof}

Next, we show that $J$ satisfies the $(PS)_c$ condition.

\begin{lm}\label{ln}
Let $(h_1),(h_2)$ and $(f_1)$-$(f_3)$ hold. Then for any $c\in\mathbb{R}$, $J$ satisfies the $(PS)_c$ condition. 
\end{lm}
\begin{proof}
For any $c\in\mathbb{R}$, let $\left\{u_{n}\right\}$ be a $(P S)_{c}$ sequence of $J$,
\begin{equation*}
 J\left(u_{n}\right)\rightarrow c, \quad \text { and } \quad J^{\prime}\left(u_{n}\right)\rightarrow0.
 \end{equation*}
By (\ref{ba}), we get that $\{u_n\}$ is bounded in $H$. Then by Lemma \ref{lgg}, passing to a subsequence if necessary, there exists $u \in H$ such that
\begin{equation}\label{ua}
\begin{cases}u_{n} \rightharpoonup u, & \text { in } H, \\ u_{n}\rightarrow u, & \text{pointwise~in~} \mathbb{Z}^N, \\ u_{n} \rightarrow u, & \text { in } \ell^{q}(\mathbb{Z}^N),q\geq p.\end{cases}
\end{equation}
Then it follows from the HLS inequality (\ref{bo}), H\"{o}lder inequality, the boundedness of $\{u_n\}$ and (\ref{ua}) that
$$
\begin{aligned}
&\int_{\mathbb{Z}^N}(R_\alpha \ast F(u_n))|f(u_n)(u_n-u)|\,d\mu
\\ \leq& C\left(\int_{\mathbb{Z}^N} |F(u_n)|^{\frac{2N}{N+\alpha}}\,d\mu  \right)^{\frac{N+\alpha}{2N}}\left(\int_{\mathbb{Z}^N} |f(u_n)(u_n-u)|^{\frac{2N}{N+\alpha}}\,d\mu  \right)^{\frac{N+\alpha}{2N}}\\ \leq&C\left(\int_{\mathbb{Z}^N} (|u_n|^p+|u_n|^\tau)^{\frac{2N}{N+\alpha}}\,d\mu  \right)^{\frac{N+\alpha}{2N}}\left(\int_{\mathbb{Z}^N} \left[(|u_n|^{p-1}+|u_n|^{\tau-1})|u_n-u|\right]^{\frac{2N}{N+\alpha}}\,d\mu  \right)^{\frac{N+\alpha}{2N}}\\ \leq& C\left[\left(\int_{\mathbb{Z}^N} (|u_n|^{p-1}|u_n-u|)^{\frac{2N}{N+\alpha}}\,d\mu  \right)^{\frac{N+\alpha}{2N}}+\left(\int_{\mathbb{Z}^N} (|u_n|^{\tau-1}|u_n-u|)^{\frac{2N}{N+\alpha}}\,d\mu  \right)^{\frac{N+\alpha}{2N}}\right] \\ \leq& C\|u_n\|^{p-1}_{\frac{2Np}{N+\alpha}}\|u_n-u\|_{\frac{2Np}{N+\alpha}}+C\|u_n\|^{\tau-1}_{\frac{2N\tau}{N+\alpha}}\|u_n-u\|_{\frac{2N\tau}{N+\alpha}} \\ \leq& C\|u_n-u\|_{\frac{2Np}{N+\alpha}}+C\|u_n-u\|_{\frac{2N\tau}{N+\alpha}} \\ \rightarrow&0,\quad n\rightarrow\infty.
\end{aligned}
$$
Therefore, one gets that
\begin{eqnarray*}
    |(u_n,u_n-u)|&\leq& \left|\langle J'(u_n),u_n-u\rangle\right|+\left|\int_{\mathbb{Z}^N}(R_\alpha \ast F(u_n))f(u_n)(u_n-u)\,d\mu\right|\\&\leq&o_n(1)\|u_n-u\|+\int_{\mathbb{Z}^N}(R_\alpha \ast F(u_n))|f(u_n)(u_n-u)|\,d\mu\\&\rightarrow& 0,\quad n\rightarrow\infty.
\end{eqnarray*}
Note that $u_n\rightharpoonup u$ in $H$, then we have $(u,u_n-u)\rightarrow 0$ as $n\rightarrow\infty$. The previous two results state that
$$\|u_n-u\|\rightarrow 0,\quad n\rightarrow\infty.$$
Then we obtain that $u_n\rightarrow u$ in $H$ since $u_n\rightarrow u$ pointwise in $\mathbb{Z}^N$.

\end{proof}

\
\

{\bf Proof of Theorem \ref{t1}. } By Lemma \ref{lm} and Lemma \ref{ln}, one sees that $J$ satisfies the geometric structure and $(PS)_{c_0}$ condition with
$$c_0=\inf\limits_{\gamma\in\Gamma}\max\limits_{s\in [0,1]} J(\gamma(s)),$$
where
$$\Gamma=\{\gamma\in C([0,1],H):\gamma(0)=0, J(\gamma(1))<0\}.$$ Then by the mountain pass theorem, there exists $u\in H$ such that $J(u)=c_0$ and $J'(u)=0$.  We get that $c_0=c>0$ by Corollary \ref{jb}. Hence $u\neq 0$ and $u\in \mathcal{M}.$ We complete the proof. \qed

\
\

{\bf Declarations}

\
\

{\bf Conflict of interest:} The author declares that there are no conflicts of interests regarding the publication of
this paper.

\end{document}